\pdfoutput=1
\documentclass[a4paper,12pt,reqno,oneside]{amsart}
\usepackage[stretch=10]{microtype}

\usepackage[utf8x]{inputenc}
\usepackage[T1]{fontenc}
\usepackage{geometry}
\usepackage{lmodern}

\usepackage{amsmath}
\usepackage{amssymb}
\usepackage{amsthm}
\usepackage[matrix,arrow,cmtip]{xy}
\usepackage{enumerate}
\PassOptionsToPackage{hyphens}{url}
\usepackage[bookmarksnumbered,colorlinks,linkcolor=black,citecolor=black,urlcolor=black]{hyperref}
\usepackage[capitalise]{cleveref}

\newtheorem{lemma}{Lemma}[section]
\newtheorem{corollary}[lemma]{Corollary}
\newtheorem{proposition}[lemma]{Proposition}
\newtheorem{theorem}[lemma]{Theorem}
\newtheorem{question}[lemma]{Question}
\Crefname{question}{Question}{Questions}
\newtheorem*{corollary*}{Corollary}

\theoremstyle{remark}
\newtheorem{remark}[lemma]{Remark}

\newcommand{\defterm}[1]{\textbf{#1}}

\newcommand{\bs}{\backslash}
\newcommand{\ov}{\overline}

\newcommand{\ad}{\mathrm{ad}}
\newcommand{\can}{\mathrm{can}}
\newcommand{\der}{\mathrm{der}}

\DeclareMathOperator{\Aut}{Aut}
\DeclareMathOperator{\Gal}{Gal}
\DeclareMathOperator{\Hom}{Hom}
\DeclareMathOperator{\Inn}{Inn}
\DeclareMathOperator{\Lie}{Lie}
\DeclareMathOperator{\MT}{MT}
\DeclareMathOperator{\Res}{Res}
\DeclareMathOperator{\Sh}{Sh}
\DeclareMathOperator{\Stab}{Stab}
\DeclareMathOperator{\spl}{sp}

\newcommand{\bA}{\mathbb{A}}
\newcommand{\bC}{\mathbb{C}}
\newcommand{\bG}{\mathbb{G}}
\newcommand{\bQ}{\mathbb{Q}}
\newcommand{\bR}{\mathbb{R}}
\newcommand{\bS}{\mathbb{S}}
\newcommand{\Qbar}{\overline \bQ}

\newcommand{\cS}{\mathcal{S}}
\newcommand{\cT}{\mathcal{T}}
\newcommand{\Hg}{\mathcal{H}_g}

\newcommand{\rH}{\mathrm{H}}

\newcommand{\gG}{\mathbf{G}}
\newcommand{\gGSp}{\mathbf{GSp}}
\newcommand{\gH}{\mathbf{H}}
\newcommand{\gM}{\mathbf{M}}
\newcommand{\gT}{\mathbf{T}}
\newcommand{\gZ}{\mathbf{Z}}

\newcommand{\thG}{{^{\tau,h}}\gG}
\newcommand{\thM}{{^{\tau,h}}\gM}
\newcommand{\thT}{{^{\tau,h}}\gT}
\newcommand{\thZ}{{^{\tau,h}}\gZ}
\newcommand{\thX}{{^{\tau,h}}X}
\newcommand{\thpG}{{^{\tau,h'}}\gG}
\newcommand{\thpX}{{^{\tau,h'}}X}
\newcommand{\thH}{{^{\tau,h}}\gH}
\newcommand{\thXH}{{^{\tau,h}}X_\gH}

\newcommand{\refConjC}{\cref{conjC}}

\title[Galois conjugates of special points and special subvarieties]{Galois conjugates of special points and special subvarieties in Shimura varieties}
\author{Martin Orr}


\begin{document}

\begin{abstract}
Let \( S \) be a Shimura variety with reflex field \( E \).
We prove that the action of \( \Gal(\Qbar/E) \) on \( S \) maps special points to special points and special subvarieties to special subvarieties.
Furthermore, the Galois conjugates of a special point all have the same complexity (as defined in the theory of unlikely intersections).
These results follow from Milne and Shih's construction of canonical models of Shimura varieties, based on a conjecture of Langlands which was proved by Borovoi and Milne.
\end{abstract}

\maketitle

\section{Introduction}

\subsection*{Special points}

Let \( S = \Sh_K(\gG, X) \)  be a Shimura variety.
Then \( S \) has a canonical model over the reflex field \( E_\gG = E(\gG, X) \).
According to the definition of a canonical model, for every special point \( s \in S \) with reflex field \( E(s) \), the Galois group \( \Gal(\Qbar/E(s)) \) acts on the Hecke orbit of \( s \) via a reciprocity morphism \cite[2.2.5]{Del79}.
In particular, if \( \tau \in \Gal(\Qbar/E(s)) \), then \( \tau(s) \) is also a special point.
In general \( E(s) \) is a non-trivial extension of \( E_\gG \), so this raises the following question.

\begin{question} \label{q:conj-sp-pt}
If \( \tau \in \Gal(\Qbar/E_\gG) \) does not fix \( E(s) \), is \( \tau(s) \) still a special point?
\end{question}

Langlands \cite{Lan79} stated a conjecture on the conjugation of Shimura varieties which implies the existence of canonical models (the construction of a canonical model, assuming the conjecture of Langlands, was completed in \cite{MS82b}).
Subsequently Borovoi \cite{Bor84} and Milne \cite{Mil83} proved the conjecture of Langlands.
This construction of canonical models gives a positive answer to \cref{q:conj-sp-pt}.

\subsection*{Special subvarieties}

We may ask a similar question about special subvarieties of dimension greater than zero.
By definition, a \defterm{special subvariety} of \( S \) is a geometrically irreducible component of the translate by a Hecke correspondence \( T_g \) of the image of a Shimura morphism \( [f] \colon \Sh_{K_\gH}(\gH, X_\gH) \to \Sh_K(\gG, X) \).
The Shimura variety \( \Sh_{K_\gH}(\gH, X_\gH) \) has a canonical model over \( E_\gH = E(\gH, X_\gH) \).
By \cite[Cor.~5.4]{Del71}, the morphism \( [f] \) is defined over the compositum \( E_\gG E_\gH \).
Consequently if \( Z \) is a component of the image of \( T_g \circ [f] \), then for every \( \tau \in \Gal(\Qbar/E_\gG E_\gH) \), \( \tau(Z) \) is again a geometrically irreducible component of the image of \( T_g \circ [f] \) and so \( \tau(Z) \) is a special subvariety of \( S \).
This leads to the following question.

\begin{question} \label{q:conj-sp-subvar}
If \( \tau \in \Gal(\Qbar/E_\gG) \) does not fix \( E_\gG E_\gH \), is \( \tau(Z) \) still a special subvariety?
\end{question}

The conjecture of Langlands tells us that \( \tau \Sh_{K_\gH}(\gH, X_\gH) \) is isomorphic to another Shimura variety, but it does not immediately tell us that the morphism \( \tau[f] \colon \tau \Sh_{K_\gH}(\gH, X_\gH) \to \Sh_K(\gG, X) \) is a Shimura morphism.
A positive answer to \cref{q:conj-sp-subvar} can be obtained from the proof of \cite[Lemma~9.5]{MS82b}.

Thus the answers to the above questions are implicit in \cite{MS82b} but they are not explicitly stated there.
The first goal of this paper is to explain enough of the machinery of \cite{MS82b} to answer \cref{q:conj-sp-pt,q:conj-sp-subvar}.

\subsection*{Complexity}

The second goal of the paper is to prove that all Galois conjugates of a special point have the same complexity.
The complexity of a special point is a quantity used in studying  questions of unlikely intersections such as the André--Oort conjecture.
The complexity of special points in general Shimura varieties was first used by Ullmo and Yafaev \cite{UY14}.
For a precise definition, we use a generalisation of \cite[Definition~10.1]{DR}.
We have generalised the definition slightly: \cite{DR} considered only a single geometrically irreducible component of a Shimura variety, and therefore could always choose \( g=1 \) (in the notation of the definition below).
We need to explicitly account for~\( g \) so that the complexity is well-defined for special points in all components of the Shimura variety.

Let \( s \) be a special point in \( \Sh_K(\gG, X) \).
The complexity of~\( s \) is defined as follows.
\begin{enumerate}
\item Write \( s = [h,g]_K \) for some \( h \in X \) and \( g \in \gG(\bA_f) \).
\item Let \( \gM \) be the Mumford-Tate group of \( h \).
This is a \( \bQ \)-torus in \( \gG \).
\item Let \( K_\gM^m \) be the maximal compact subgroup of \( \gM(\bA_f) \).
There is a unique maximal compact subgroup because \( \gM \) is a torus.
\item Let \( K_\gM = gKg^{-1} \cap \gM(\bA_f) \).
This is a compact subgroup, so contained in \( K_\gM^m \).
\item Let \( D_\gM \) be the absolute value of the discriminant of the splitting field of~\( \gM \).
\item Let \( \Delta(s) = \max\{D_\gM, \, [K_\gM^m : K_\gM]\} \).
\end{enumerate}
The \defterm{complexity} of \( s \) is \( \Delta(s) \).

If we make a different choice of \( (h',g') \in X \times \gG(\bA_f) \) lifting \( s \), then \( h' = \gamma h \) for some \( \gamma \in \gG(\bQ) \) such that \( g'^{-1} \gamma g \in K \).
Writing \( \gM' = \MT(h') \) and \( K_{\gM'} = g' K g'^{-1} \cap \gM'(\bA_f) \), we get \( \gM' = \gamma \gM \gamma^{-1} \) and \( K_{\gM'} = \gamma K_{\gM} \gamma^{-1} \).
Hence \( \Delta(s) \) is independent of the choice of \( (h,g) \).

Our main result on complexity is the following.

\begin{theorem} \label{galois-complexity}
Let \( S = \Sh_K(\gG, X) \) be the canonical model of a Shimura variety over the reflex field \( E_\gG = E(\gG, X) \).
Let \( s \in S \) be a special point.
Then for every \( \tau \in \Gal(\Qbar/E_\gG) \),  we have \( \Delta(\tau(s)) = \Delta(s) \).
\end{theorem}

The proof of \cref{galois-complexity} uses details of Milne and Shih's construction of descent data for Shimura varieties.
As with the answers to \cref{q:conj-sp-pt,q:conj-sp-subvar}, the theorem has a simpler proof when restricted to \( \tau \) fixing \( E(s) \): see \cite[p.~156]{Daw15}.

\subsection*{Motivation: unlikely intersections}

The motivation for this paper came from work on unlikely intersections in collaboration with Christopher Daw.
One aim of the paper is to give full proofs of certain claims in \cite{DR} which are well-known to experts but for which either no proof appears in the literature,
% (in the case of \cref{galois-complexity})
or the proof can be found only within the proof of a larger result and the claim used in \cite{DR} is never explicitly stated.
% (in the case of \cref{q:conj-sp-subvar})
In particular, the last paragraph of \cite[p.~1869]{DR} claims that the answer to \cref{q:conj-sp-subvar} is positive.
The same paragraph also claims that
\[ \Delta(\sigma(Z)) = \Delta(Z) \]
where \( Z \) is a special subvariety of a Shimura variety component \( S \), \( \sigma \) is an element of \( \Gal(\Qbar/F) \) (where \( F \) is a field of definition for \( S \)), and \( \Delta \) is defined as
\[ \Delta(Z) = \max \{ \deg(Z), \min\{ \Delta(P) : P \in Z \text{ is a special point} \} \}. \]
The equality \( \Delta(\sigma(Z)) = \Delta(Z) \) is an immediate corollary of \cref{galois-complexity}.

The results of this paper will also be used in a forthcoming paper of Daw and the author on unlikely intersections with Hecke--facteur families.

\subsection*{Outline of paper}

Sections~\ref{sec:svs}--\ref{sec:langlands} recall various known facts and definitions, in order to make our terminology and notation clear and to gather together in one place all the facts we will use.
Section~\ref{sec:svs} consists of definitions of Shimura varieties and associated objects.
Most of this is standard, except our use of the term ``Shimura pro-variety'' for the inverse limit of the system of Shimura varieties associated with a given Shimura dataum.
Section~\ref{sec:serre-taniyama} outlines key facts about the Serre and Taniyama groups from \cite{Lan79} and \cite{MS82a}.
Section~\ref{sec:langlands} states the conjecture of Langlands on conjugation of Shimura varieties, which is central to all the results of the paper.
It also explains the construction of the twisted group \( \thG \) which appears in this conjecture, based on \cite{Lan79} and \cite{MS82b}.

In section~\ref{sec:conjugation} we prove that \cref{q:conj-sp-pt,q:conj-sp-subvar} have positive answers.
This is a simple application of the conjecture and construction in section~\ref{sec:langlands}.
Finally in section~\ref{sec:conj-complexity} we prove \cref{galois-complexity} on the complexity of Galois conjugates of special points.
This depends on further details of the construction from \cite{MS82b} as well as a lemma on morphisms of Shimura pro-varieties, which we prove.

\subsection*{Acknowledgements}

The questions discussed in this paper arose during collaboration with Christopher Daw.  I am grateful to him for extensive discussion of these questions and for carefully reading drafts of this paper.  I am also grateful to Andrei Yafaev for helpful discussions.
I would like to thank the anonymous referee for careful reading of the paper and suggesting some improvements to the clarity of explanations.

\section{Preliminaries: Shimura varieties} \label{sec:svs}

We recall various definitions associated with Shimura varieties, in order to fix terminology and notation.

A \defterm{Shimura datum} is a pair \( (\gG, X) \) where \( \gG \) is a connected reductive \( \bQ \)-algebraic group and \( X \) is a \( \gG(\bR) \)-conjugacy class in \( \Hom(\bS, \gG_\bR) \) such that each \( h \in X \) satisfies the following axioms \cite[2.1.1.1--2.1.1.3]{Del79}.
(Here \( \bS \) denotes the Deligne torus \( \Res_{\bC/\bR} \bG_m \).)
\begin{enumerate}
\item The Hodge structure on \( \Lie(\gG_\bR) \) induced by \( h \) (via the adjoint representation of \( \gG \)) has type \( \{ (-1,1), (0,0), (1,-1) \} \).
\item The involution \( \operatorname{Int} h(i) \) is a Cartan involution of the adjoint group \( \gG^\ad_\bR \).
\item \( \gG^\ad \) has no \( \bQ \)-simple factor on which the projection of \( h \) is trivial.
\end{enumerate}
These axioms imply that \( X \) is a finite disjoint union of Hermitian symmetric domains \cite[Cor.~1.1.17]{Del79}.

Given a Shimura datum \( (\gG, X) \) and a compact open subgroup \( K \subset \gG(\bA_f) \), we can form a quasi-projective complex variety \( \Sh_K(\gG, X) \) whose \( \bC \)-points are
\[ \Sh_K(\gG, X)(\bC) = \gG(\bQ) \mathop\bs X \times \gG(\bA_f) \mathop/ K. \]
Here \( \gG(\bQ) \) acts diagonally on \( X \times \gG(\bA_f) \) on the left, while \( K \) acts only on \( \gG(\bA_f) \) on the right.
We call \( \Sh_K(\gG, X) \) a \defterm{Shimura variety}.
We write \( [h,g]_K \) for the complex point of \( \Sh_K(\gG, X) \) which is represented by \( (h,g) \in X \times \gG(\bA_f) \).

\medskip

If \( K' \subset K \), then there is a finite morphism \( \Sh_{K'}(\gG, X) \to \Sh_K(\gG, X) \).
Thus the Shimura varieties \( \Sh_K(\gG, X) \) form a projective system as \( K \) varies over compact open subgroups of \( \gG(\bA_f) \).
The inverse limit of this system is a scheme over \( \bC \), not of finite type, which we denote \( \Sh(\gG, X) \).
Its \( \bC \)-points are given by
\begin{equation} \label{eqn:spv-c-pts}
\Sh(\gG, X)(\bC) = \gG(\bQ) \mathop\bs X \times \gG(\bA_f) \mathop/ \ov{\gZ(\bQ)}.
\end{equation}
Here \( \gZ \) denotes the centre of \( \gG \) and \( \ov{\gZ(\bQ)} \) is the closure of \( \gZ(\bQ) \) in the adelic topology on \( \gZ(\bA_f) \).
Again \( \gG(\bQ) \) acts diagonally on \( X \times \gG(\bA_f) \) on the left, while \( \ov{\gZ(\bQ)} \) acts only on \( \gG(\bA_f) \) on the right.
Note that \( \gZ(\bQ) \) acts trivially on both \( X \) and \( \gG(\bA_f) / \ov{\gZ(\bQ)} \), so \eqref{eqn:spv-c-pts} is equivalent to the description in \cite[Prop.~2.1.10]{Del79}:
\[ \Sh(\gG, X)(\bC) = \raisebox{-0.3em}{\( \bigl( \gG(\bQ) \mathop/ \gZ(\bQ) \bigr) \)} \bigm\bs \raisebox{0.3em}{\( X \times \bigl( \gG(\bA_f) \mathop/ \ov{\gZ(\bQ)} \bigr) \)}. \]
We call \( \Sh(\gG, X) \) a \defterm{Shimura pro-variety}.
We write \( [h,g] \) for the complex point of \( \Sh(\gG, X) \) which is represented by \( (h,g) \in X \times \gG(\bA_f) \).

\medskip

If \( g \in \gG(\bA_f) \), write \( T_g \colon \Sh(\gG, X) \to \Sh(\gG, X) \) for the morphism of pro-varieties
\[ T_g [h,q] = [h,qg]. \]
This gives a right action of \( \gG(\bA_f) \) on \( \Sh(\gG, X) \).
The morphisms \( T_g \) are known as \defterm{Hecke operators}.
The orbit of any point in \( \Sh(\gG, X) \) under the action of \( \gG(\bA_f) \) is called a \defterm{Hecke orbit}.

Let \( K \) be a compact open subgroup of \( \gG(\bA_f) \) and \( g \in \gG(\bA_f) \).
The Hecke operator \( T_g \colon \Sh(\gG, X) \to \Sh(\gG, X) \) induces a morphism of varieties
\[ \cdot g \colon \Sh_{K \cap gKg^{-1}}(\gG, X) \to \Sh_{g^{-1}Kg \cap K}(\gG, X). \]
Hence the following diagram defines a correspondence on \( \Sh_K(\gG, X) \):
\[ \xymatrix{
%    \Sh(\gG, X)                    \ar[d]  \ar[r]^{T_g}
%  & \Sh(\gG, X)                    \ar[d]
    \Sh_{K \cap gKg^{-1}}(\gG, X)  \ar[d]  \ar[r]^{\cdot g}
  & \Sh_{g^{-1}Kg \cap K}(\gG, X)  \ar[d]
\\  \Sh_K(\gG, X)
  & \Sh_K(\gG, X)
} \]
We call this correspondence a \defterm{Hecke correspondence} and also denote it by \( T_g \).

A \defterm{morphism of Shimura data} \( (\gG_1, X_1) \to (\gG_2, X_2) \) is a homomorphism of \( \bQ \)-algebraic groups \( f \colon \gG_1 \to \gG_2 \) such that composition with \( f \) maps \( X_1 \) into \( X_2 \).
A morphism of Shimura data induces a morphism of pro-varieties
\[ [f] \colon \Sh(\gG_1, X_1) \to (\gG_2, X_2). \]
If \( K_1 \subset \gG_1(\bA_f) \) and \( K_2 \subset \gG_2(\bA_f) \) are compact open subgroups such that \( f(K_1) \subset K_2 \), then \( f \) also induces a morphism
\[ [f] \colon \Sh_{K_1}(\gG_1, X_1) \to \Sh_{K_2}(\gG_2, X_2). \]
We call either of these induced morphisms \( [f] \) a \defterm{Shimura morphism}.

A \defterm{Shimura subdatum} of \( (\gG, X) \) is a Shimura datum \( (\gH, X_\gH) \) where \( \gH \subset \gG \) and \( X_\gH \subset X \), with the inclusions \( \gH \hookrightarrow \gG \) and \( X_\gH \hookrightarrow X \) being compatible.
The inclusion of Shimura data induces a morphism of Shimura pro-varieties \( \Sh(\gH, X_\gH) \to \Sh(\gG, X) \), which is a closed immersion by \cite[Prop.~1.15]{Del71}.
We call the image of \( \Sh(\gH, X_\gH) \to \Sh(\gG, X) \) a \defterm{Shimura sub-pro-variety} of \( \Sh(\gG, X) \).

\medskip

We will next recall Deligne's definition of a canonical model of a Shimura pro-variety.
Before this we need to recall certain other definitions.

For any point \( h \in X \), the \defterm{Mumford--Tate group} of \( h \) is the smallest \( \bQ \)-algebraic subgroup \( \MT(h) \subset \gG \) such that \( h \) factors through \( \MT(h)_\bR \).
The \defterm{generic Mumford--Tate group} of \( (\gG, X) \) is the smallest \( \bQ \)-algebraic subgroup \( \MT(X) \subset \gG \) such that \emph{every} \( h \in X \) factors through \( \MT(X)_\bR \).
A point \( h \in X \) is said to be \defterm{Hodge generic} if \( \MT(h) = \MT(X) \).
It is well-known that every Shimura datum contains Hodge generic points; see for example the proof of \cite[Prop.~7.5]{Del72}.

A \defterm{pre-special point} is a point \( h \in X \) for which \( \MT(h) \) is commutative.
Mumford--Tate groups are always reductive, so this implies that \( \MT(h) \) is a torus.
A \defterm{special point} is a point \( [h,g] \in \Sh(\gG, X) \) or \( [h,g]_K \in \Sh_K(\gG, X) \) such that \( h \) is pre-special (note that this is independent of the choice of lift \( (h,g) \)).

\medskip

Let \( \mu \colon \bG_{m,\bC} \to \bS_\bC \) denote the cocharacter given by \( \mu(z) = (z,1) \) (identifying \( \bS(\bC) \) with \( \bC^\times \times \bC^\times \)).
If \( h \in X \), then the \( \gG(\bC) \)-conjugacy class of \( h \circ \mu \colon \bS_\bC \to \gG_\bC \) is defined over a number field.
We call the field of definition of this conjugacy class (inside \( \bC \)) the \defterm{reflex field} of the Shimura datum \( (\gG, X) \) and write it \( E(\gG, X) \).

If \( h \in X \) is a pre-special point with Mumford--Tate group \( \gM \subset \gG \), then the pair \( (\gM, \{h\}) \) is a Shimura datum.
We define \( E(h) \), the \defterm{reflex field} of \( h \), to be the reflex field of \( (\gM, \{h\}) \).
Note that \( E(\gG, X) \subset E(h) \).
The complex points of \( \Sh(\gM, \{h\}) \) form a pro-finite set.
Deligne \cite[2.2.4]{Del79} defined an action of \( \Gal(\Qbar/E(h)) \) on \( \Sh(\gM, \{h\})(\bC) \)
by means of a reciprocity morphism \( \Gal(\Qbar/E(h)) \to \gM(\bA_f) / \ov{\gM(\bQ)} \).
The pro-variety \( \Sh(\gM, \{h\}) \) has a unique model over \( E(h) \) for which the Galois action is the same as the reciprocity action.

A \defterm{canonical model} of a Shimura pro-variety \( \Sh(\gG, X) \) is a scheme \( M \) over \( E(\gG, X) \) equipped with a right action of \( \gG(\bA_f) \) and a \( \gG(\bA_f) \)-equivariant isomorphism \( M \times_{E(\gG, X)} \bC \to \Sh(\gG, X) \) such that
\begin{enumerate}[(a)]
\item the special points in \( M \) are defined over \( \Qbar \);
\item for each pre-special point \( h \in X \), the morphism \( \Sh(\MT(h), \{h\}) \to \Sh(\gG, X) \) induced by the inclusion \( \MT(h) \hookrightarrow \gG \) is defined over \( E(h) \).
%the action of \( \Gal(\Qbar/E(\gG, X)) \) on \( M(\Qbar) \) restricts to the reciprocity action on the image of \( \{h\} \times \gG(\bQ) \).
\end{enumerate}
According to \cite[Cor.~5.5]{Del71}, a Shimura pro-variety has at most one canonical model (up to unique isomorphism).
According to \cite[Cor.~5.4]{Del71}, if \( \Sh(\gG_1, X_1) \) and \( \Sh(\gG_2, X_2) \) have canonical models, then any Shimura morphism \( \Sh(\gG_1, X_1) \to \Sh(\gG_2, X_2) \) is defined over the compositum of the reflex fields \( E(\gG_1, X_1).E(\gG_2, X_2) \).

Deligne (\cite{Del71} and \cite{Del79}) established the existence of canonical models for a large class of Shimura pro-varieties, namely those of ``abelian type'', starting from the fact that the moduli space of principally polarised abelian varieties of dimension~\( g \) is defined over \( \bQ \) and this gives a canonical model for \( \Sh(\gGSp_{2g}, \Hg^\pm) \).
Milne and Shih \cite{MS82b} proved that a conjecture of Langlands implies the existence of canonical models for all Shimura pro-varieties.
Borovoi \cite{Bor84} and Milne \cite{Mil83} then proved the conjecture of Langlands using a result of Kazhdan \cite{Kaz83}, completing the proof of the existence of canonical models.

\section{The Serre and Taniyama groups} \label{sec:serre-taniyama}

We recall some facts about the Serre and Taniyama groups which are required in order to set up the conjecture of Langlands.

The Serre group \( \cS \) is a pro-algebraic torus over \( \bQ \) (i.e.\ an inverse limit of a projective system of \( \bQ \)-tori) which can be thought of as the ``universal Mumford--Tate group of a \( \bQ \)-rational polarisable Hodge structure of CM type.''
More formally, \( \cS \) is the Tannakian group of the category of \( \bQ \)-rational polarisable Hodge structures of CM type (with the obvious forgetful fibre functor to \( \bQ \)-vector spaces).

An explicit construction of \( \cS \) is described in \cite[\S 1]{MS82a}, as well as the construction of a canonical Hodge parameter \( h_{\can} \colon \bS \to \cS_\bR \).
We shall need the following universal property of \( (\cS, h_{\can}) \).

\begin{lemma} \label{serre-universal}
For every \( \bQ \)-algebraic torus \( \gT \) and every \( h \colon \bS \to \gT_\bR \), if the weight homomorphism \( h \circ w \colon \bG_{m,\bR} \to \gT_\bR \) is defined over \( \bQ \), then there exists a unique homomorphism of pro-\( \bQ \)-algebraic tori \( \rho \colon \cS \to \gT \) such that \( h = \rho \circ h_{\can} \).
\end{lemma}

Here \( w \colon \bG_{m,\bR} \to \bS \) denotes the morphism given on \( \bR \)-points by the inclusion \( \bR^\times \to \bC^\times \).
The condition that \( h \circ w \) is defined over \( \bQ \) is equivalent to \cite[(1.1)]{MS82a}.
The universal property determines \( (\cS, h_{\can}) \) up to unique isomorphism.

\medskip

The Taniyama group \( \cT \) is an extension of \( \Gal(\Qbar/\bQ) \) by \( \cS \) which was defined by Langlands \cite[\S 5]{Lan79}.
According to \cite{Del82}, it is isomorphic to the Tannakian group of the category of absolute Hodge CM motives over \( \bQ \).
The Taniyama group comes with an exact sequence
\[ 1 \longrightarrow \cS \longrightarrow \cT \overset{\pi}{\longrightarrow} \Gal(\Qbar/\bQ) \longrightarrow 1. \]
This is an exact sequence of pro-\( \bQ \)-algebraic groups if we make \( \Gal(\Qbar/\bQ) \) into a pro-\( \bQ \)-group by regarding it as a limit of finite groups \( \Gal(L/\bQ) \) and declaring that every point of these finite sets is a \( \bQ \)-point.
The Taniyama group is also equipped with a splitting of the exact sequence over \( \bA_f \):
\[ \spl \colon \Gal(\Qbar/\bQ) \to \cT(\bA_f). \]

For any \( \tau \in \Gal(\Qbar/\bQ) \), \( \pi^{-1}(\tau) \subset \cT \) is a pro-\( \bQ \)-variety.
Letting \( \cS \) act on \( \pi^{-1}(\tau) \) by multiplication on the right, we get a right \( \cS \)-torsor which we denote \( {^\tau \cS} \) (in choosing the right action, we are following \cite[Remark~2.9]{MS82a}).
Thanks to \( \spl \), this \( \cS \)-torsor is split over \( \bA_f \).

\section{The conjecture of Langlands} \label{sec:langlands}

Let \( (\gG, X) \) be a Shimura datum.
In order to give a conjectural description of the Galois conjugates of \( \Sh(\gG, X) \), Langlands \cite[\S 6]{Lan79} constructed the following objects for each pre-special point \( h \in X \) and each \( \tau \in \Gal(\Qbar/\bQ) \):
\begin{enumerate}[(i)]
\item a Shimura datum \( (\thG, \thX) \);
\item a pre-special point \( {^\tau h} \in \thX \);
\item a continuous group isomorphism \( \theta_{\tau,h} \colon \gG(\bA_f) \to \thG(\bA_f) \).
\end{enumerate}
The construction depends on the chosen pre-special point \( h \in X \).
In order to explain how the resulting Shimura pro-varieties are related when we vary \( h \),
Langlands also constructed an isomorphism of pro-\( \bC \)-varieties
\[ \phi(\tau;h',h) \colon \Sh(\thG, \thX) \to \Sh(\thpG, \thpX) \]
for each pair of pre-special points \( h, h' \in X \).

\begin{remark}
Our notation is based on \cite{MS82b}, which differs slightly from the notation in \cite{Lan79} in the positioning of superscripts.
We always explicitly include the dependence on \( h \) in our notation, while both \cite{Lan79} and \cite{MS82b} frequently omit it.
We label various objects with the Hodge parameter \( h \colon \bS \to \gG_\bR \), while \cite{Lan79} and \cite{MS82b} use the cocharacter \( \mu \colon \bG_{m,\bC} \to \gG_\bC \); since \( h \in X \) and \( \mu \) determine each other, this does not matter.
The isomorphism of adelic groups \( \theta_{\tau,h} \) is not given a name in \cite{Lan79} or \cite{MS82b}, being denoted simply by \( g \mapsto g^\tau \) or \( g \mapsto {^\tau g} \) respectively, but we have found it convenient to name it explicitly.
\end{remark}

Before outlining the construction of these objects, we shall first state Conjecture~C of Langlands and discuss its consequences for canonical models.
The original statement of this conjecture was at \cite[pp.~232--233]{Lan79}.
Alternative formulations can be found at \cite[p.~311]{MS82b} and \cite[\S 2]{Bor82}.
The conjecture was proved by Borovoi \cite{Bor84} (completed in \cite{Bor87}) and Milne \cite[Thm.~7.1]{Mil83}.

% \begin{conjectureC}  \hypertarget{conjectureC}{}
\begin{theorem} \label{conjC}
Let \( (\gG, X) \) be a Shimura datum and let \( \tau \in \Aut(\bC) \).
\begin{enumerate}[(a)]
\item For every pre-special point \( h \in X \), there is an isomorphism of pro-\( \bC \)-varieties
\[ \phi_{\tau,h} \colon \tau \Sh(\gG, X) \to \Sh(\thG, \thX) \]
such that
\begin{enumerate}[(i)]
\item \( \phi_{\tau,h}(\tau[h,1]) = [{^\tau h}, 1] \); and
\item the diagram
\vspace{-6pt}
\[ \renewcommand{\labelstyle}{\textstyle}
\xymatrix@C+3pc@R+1pc{
    \tau \Sh(\gG, X)  \ar[r]^{\tau T_g}  \ar[d]^{\phi_{\tau,h}}
  & \tau \Sh(\gG, X)                     \ar[d]^{\phi_{\tau,h}}
\\  \Sh(\thG, \thX)   \ar[r]^{T_{\theta_{\tau,h}(g)}}
  & \Sh(\thG, \thX)
} \]
commutes for every \( g \in \gG(\bA_f) \).
In other words,
\[ \phi_{\tau,h}(\tau(T_g(s))) = T_{\theta_{\tau,h}(g)}(\phi_{\tau,h}(\tau(s))) \]
for all \( s \in \Sh(\gG, X) \) and \( g \in \gG(\bA_f) \).
\end{enumerate}

\medskip

\item For every pair of pre-special points \( h, h' \in X \), the following diagram commutes:
\[
\renewcommand{\labelstyle}{\textstyle}
\xymatrix@C+2pc@R+1pc{
    \tau \Sh(\gG, X)  \ar[r]^-{\phi_{\tau,h}}  \ar[rd]_{\phi_{\tau,h'}}
  & \Sh(\thG, \thX)   \ar[d]^{\phi(\tau;h',h)}
\\& \Sh(\thpG, \thpX).
} \]
\end{enumerate}
% \end{conjectureC}
\end{theorem}

Note that the isomorphism \( \phi_{\tau,h} \) in \refConjC(a) is unique because (a)(i) and~(ii) specify what it does on the Hecke orbit of \( [h,1] \), which is dense in \( \Sh(\gG, X) \).

If \( \tau \in \Aut(\bC) \) fixes \( E := E(\gG, X) \), then Milne and Shih \cite[Remark~4.13 and Prop.~7.8]{MS82b} construct an isomorphism of pro-\( \bC \)-varieties
\[ \phi(\tau;h) \colon \Sh(\gG, X) \to \Sh(\thG, \thX) \]
such that
\begin{enumerate}[(i)]
\item \( \phi(\tau;h) \circ T_g = T_{\theta_{\tau,h}(g)} \circ \phi(\tau;h) \) \label{phi-tau-h}
for all \( g \in \gG(\bA_f) \);
\item \( \phi(\tau;h',h) = \phi(\tau;h') \circ \phi(\tau;h)^{-1} \)
for all pre-special points \( h, h' \in X \).
\end{enumerate}
The isomorphism \( \phi(\tau;h) \) has the form \( [f_1] \circ T_{\beta_1} \) for some isomorphism of Shimura data \( f_1 \colon (\gG, X) \to (\thG, \thX) \) and some element \( \beta_1 \in \gG(\bA_f) \).
(Note that \( \beta_1 \) and \( f_1 \) depend on choices made during the construction, but \( \phi(\tau;h) \) does not.)

Following \cite[p.~233]{Lan79}, Milne and Shih show that the morphisms
\begin{equation} \label{eqn:psit}
\psi_\tau = \phi(\tau;h)^{-1} \circ \phi_{\tau,h} \colon \tau \Sh(\gG, X) \to \Sh(\gG, X)
\end{equation}
form a descent datum (for any pre-special point~\( h \in X \)).
For a proof that this descent datum is effective, see \cite{Mil99}.
Hence there exists a model \( M(\gG, X) \) for \( \Sh(\gG, X) \) over \( E(\gG, X) \) which splits this descent datum.
In other words, \( M(\gG, X) \) is a pro-variety over \( E(\gG, X) \) with an isomorphism \( \iota \colon M(\gG, X) \times_E \bC \to \Sh(\gG, X) \) such that the following diagram commutes for every \( \tau \in \Aut(\bC/E) \):

\begin{equation} \label{eqn:descent}
\vcenter{\xymatrix{
  &  M(\gG, X) \times_E \bC  \ar[ld]_{\tau \iota}  \ar[rd]^{\iota}
\\   \tau \Sh(\gG, X)        \ar[rr]^{\psi_\tau}   \ar[rd]_{\phi_{\tau,h}}
  && \Sh(\gG, X)                                   \ar[ld]^{\phi(\tau;h)}
\\&  \Sh(\thG, \thX)
}}
\end{equation}
Because \( \psi_\tau \circ \tau T_g = T_g \circ \psi_\tau \) for all \( g \in \gG(\bA_f) \) and \( \tau \in \Aut(\bC/E) \), the right action of \( \gG(\bA_f) \) on \( \Sh(\gG, X) \) also descends to \( M(\gG, X) \).
The model \( M(\gG, X) \) is canonical by \cite[Prop.~7.14]{MS82b}.

\medskip

Now we shall outline the construction of the Shimura datum \( (\thG, \thX) \), following the description at \cite[p.~310]{MS82b}.

Let \( \gT \) be a maximal \( \bQ \)-torus in \( \gG \) such that \( h \) factors through \( \gT_\bR \).
Let \( \gT^\ad \) denote the image of \( \gT \) in \( \gG^\ad \) and let \( h^\ad \colon \bS \to \gT^\ad_\bR \) be the composition of \( \gT \to \gT^\ad \) with~\( h \).
Now \( \gT^\ad(\bR) \) is compact by axiom \cite[2.1.1.2]{Del79}, so the weight homomorphism \( h^\ad \circ w \colon \bG_{m,\bR} \to \gT^\ad_\bR \) is trivial and \textit{a fortiori} defined over \( \bQ \).
Hence we can apply the universal property of the Serre group (\cref{serre-universal}) to get a homomorphism of pro-\( \bQ \)-algebraic groups \( \rho_h \colon \cS \to \gT^\ad \) such that \( h^\ad = \rho_h \circ h_{\can} \).

Composing \( \rho_h \) with the inclusion \( \gT^\ad \to \gG^\ad \), we get a homomorphism \( \cS \to \gG^\ad \).
Now \( \gG^\ad \) acts on \( \gG \) by inner automorphisms, so we get a left action of \( \cS \) on \( \gG \) defined over \( \bQ \).
This action is independent of the choice of maximal torus \( \gT \) because it factors through \( \MT(h^\ad) \).

We define \( \thG \) to be the twist of \( \gG \) (with the left action of \( \cS \) via \( \rho_h \) which we just described) by the right \( \cS \)-torsor \( {^\tau \cS} = \pi^{-1}(\tau) \) which we defined in section~\ref{sec:serre-taniyama}:
\[ \thG = {^\tau \cS} \times^\cS \gG. \]
As remarked in \cite{MS82b}, using \cite[Remarks~2.9, 3.18]{MS82a}, there is an isomorphism \( \thG_L \cong \gG_L \) where \( L \) is a splitting field for \( \gT \).

\begin{remark} \label{thT}
Because \( \cS \) acts on \( \gG \) by inner automorphisms and the action factors through \( \gT^\ad \), the action is trivial on \( \gT \).
Hence \( {^\tau \cS} \times^\cS \gT = \gT \).
Thus \( \gT = \thT \) is naturally a subtorus of \( \thG \) and \( h \colon \bS \to \gT_\bR \) induces \( {^\tau h} \colon \bS \to \thG_\bR \).
\end{remark}

Let \( \thX \) be the \( \thG(\bR) \)-conjugacy class of \( {^\tau h} \).
Langlands showed that \( (\thG, \thX) \) is a Shimura datum \cite[p.~231]{Lan79}.
By construction, \( {^\tau h} \) factors through the \( \bQ \)-torus \( \thT \), so it is a pre-special point.

Recall that the Taniyama group comes with an adelic splitting \( \spl \colon \Gal(\Qbar/\bQ) \to \cT(\bA_f) \).
Since \( \spl(\tau) \in {^\tau \cS}(\bA_f) \), we can define a continuous group isomorphism \( \theta_{\tau,h} \colon \gG(\bA_f) \to \thG(\bA_f) = {^\tau \cS}(\bA_f) \times^\cS \gG(\bA_f) \) by
\[ \theta_{\tau,h}(g) = \spl(\tau).g. \]
Since \( \cS \) acts trivially on \( \gT \), we have \( \spl(\tau).g = g \) for all \( g \in \gT(\bA_f) \) and thus \( \theta_{\tau,h} \) restricts to the identity on \( \gT(\bA_f) = \thT(\bA_f) \).

\begin{lemma} \label{th-subdatum}
Let \( (\gH, X_\gH) \) be a Shimura subdatum of \( (\gG, X) \).
Let \( h \in X_\gH \subset X \) be a pre-special point.
Then \( (\thH, \thXH) \) is a Shimura subdatum of \( (\thG, \thX) \) and \( \theta_{\tau,h,\gH} \colon \gH(\bA_f) \to \thH(\bA_f) \) is the restriction of \( \theta_{\tau,h,\gG} \colon \gG(\bA_f) \to \thG(\bA_f) \).
\end{lemma}

\begin{proof}
Choose maximal \( \bQ \)-tori \( \gT_\gG \subset \gG \) and \( \gT_\gH \subset \gH \) such that \( \gT_\gH \subset \gT_\gG \) and \( h \) factors through \( \gT_{\gH,\bR} \).
Letting \( \gT_\gG^\ad = \gT/Z(\gG) \subset \gG^\ad \) and \( \gT_\gH^\ad = \gT/Z(\gH) \subset \gH^\ad \), we have the following commutative diagram of \( \bQ \)-tori:
\[ \xymatrix@R-2pt{
    \MT(h)                        \ar@{->>}[r] \ar@{^{(}->}[d]
  & \MT(h) / Z(\gG) \cap \MT(h)   \ar@{->>}[r] \ar@{^{(}->}[d]
  & \MT(h) / Z(\gH) \cap \MT(h)                \ar@{^{(}->}[d]
\\  \gT_\gH                       \ar@{->>}[r] \ar@{^{(}->}[d]
  & \gT_\gH / Z(\gG) \cap \gH     \ar@{->>}[r] \ar@{^{(}->}[d]
  & \gT_\gH^\ad
\\  \gT_\gG                       \ar@{->>}[r]
  & \gT_\gG^\ad
} \]
The image of \( h \) in \( \MT(h) / Z(\gG) \cap \MT(h) \) has trivial weight, so by \cref{serre-universal} it factors as \( \rho_{h,\MT} \circ h_{\can} \) for some \( \rho_{h,\MT} \colon \cS \to \MT(h) / Z(\gG) \cap \MT(h) \).
The homomorphisms \( \rho_{h,\gG} \colon \cS \to \gT_\gG^\ad \) and \( \rho_{h,\gH} \colon \cS \to \gT_\gH^\ad \) used to construct \( \thG \) and \( \thH \) respectively both factor through \( \rho_{h,\MT} \).
Hence the action of \( \cS \) on \( \gH \) coming from \( h \) is the restriction of the action of \( \cS \) on \( \gG \) coming from \( h \).
Consequently
\[ \thH = {^\tau\cS} \times^\cS \gH \subset {^\tau\cS} \times^\cS \gG = \thG. \]
Furthermore \( {^\tau h} \) is the same whether we construct it using \( \gG \) or \( \gH \).
Consequently \( (\thH, \thXH) \) is a Shimura subdatum of \( (\thG, \thX) \).

To see that \( \theta_{\tau,h,\gH} \) is the restriction of \( \theta_{\tau,h,\gG} \) to \( \gH \), simply note that both maps have the form \( g \mapsto \spl(\tau).g \).
\end{proof}

\section{Conjugation of special points and special subvarieties} \label{sec:conjugation}

In this section, we use \refConjC{} and the construction in \cref{sec:langlands} to obtain positive answers to \cref{q:conj-sp-pt,q:conj-sp-subvar}.

\begin{theorem}
Let \( S \) be the canonical model of a Shimura variety \( Sh_K(\gG, X) \).
Let \( s \in S(\Qbar) \) be a special point.
Then for every \( \tau \in \Gal(\Qbar/E(\gG, X)) \), the Galois conjugate \( \tau(s) \) is a special point of \( S \).
\end{theorem}

\begin{proof}
By \eqref{eqn:descent}, we have
\begin{equation} \label{eqn:phi-iota}
\phi(\tau;h) \circ \iota(\tau(s)) = \phi_{\tau,h} \circ (\tau\iota)(\tau(s)) = \phi_{\tau,h}(\tau(\iota(s))).
\end{equation}
By construction, \( \phi(\tau;h) \) is the composition of a Hecke operator with a Shimura isomorphism.
Hence \( \phi(\tau;h)^{-1} \) maps special points to special points.
So in order to show that \( \tau(s) \) is special, it suffices to show that
\( \phi_{\tau,h}(\tau(\iota(s))) \) is special.

Write \( \iota(s) = [h,g]_K \in \Sh_K(\gG, X)(\bC) \).
Using \refConjC(a) (ii) then~(i) gives
\begin{align}
    \phi_{\tau,h}(\tau[h,g])
%  & = \phi_{\tau,h} \circ \tau T_g(\tau[h,1])               \notag
  & = T_{\theta_{\tau,h}(g)} \circ \phi_{\tau,h}(\tau[h,1]) \notag
\\& = T_{\theta_{\tau,h}(g)} ([{^\tau h}, 1])
    = [{^\tau h}, \theta_{\tau,h}(g)].                      \label{eqn:phi-hg}
\end{align}
By construction, \( {^\tau h} \) is pre-special, so \( \phi_{\tau,h}(\tau[h,g]) \) is special.
\end{proof}

\begin{theorem} \label{sp-subpro}
Let \( (\gG, X) \) be a Shimura datum.
Let \( \Sh(\gH, X_\gH) \) be a Shimura sub-pro-variety of \( \Sh(\gG, X) \).
Then for every \( \tau \in \Gal(\Qbar/E(\gG, X)) \), \( \tau \Sh(\gH, X_\gH) \) is also a Shimura sub-pro-variety of \( \Sh(\gG, X) \).
\end{theorem}

\begin{proof}
Choose a pre-special point \( h \in X_\gH \subset X \).
Let
\begin{gather*}
   \phi_{\tau,h,\gG} \colon \tau \Sh(\gG, X) \to \Sh(\thG, \thX),
\\ \phi_{\tau,h,\gH} \colon \tau \Sh(\gH, X_\gH) \to \Sh(\thH, \thXH)
\end{gather*}
be the isomorphisms of \refConjC{}(a) for \( (\gG, X) \) and \( (\gH, X_\gH) \) respectively.

By \eqref{eqn:phi-iota} and the subsequent remark, it suffices to show that \( \phi_{\tau,h,\gG}(\tau \Sh(\gH, X_\gH)) \) is a Shimura sub-pro-variety of \( \Sh(\thG, \thX) \).
By \cref{th-subdatum}, \( (\thH, \thXH) \) is a Shimura subdatum of \( (\thG, \thX) \) so \( \Sh(\thH, \thXH) \) is a Shimura sub-pro-variety of \( \Sh(\thG, \thX) \).
Thus it suffices to show that the following diagram commutes. % (where the vertical arrows are inclusions).
\begin{equation} \label{inclusion-commutes}
\vcenter{\xymatrix@C+2pc{
    \tau \Sh(\gH, X_\gH)  \ar[r]^{\phi_{\tau,h,\gH}}  \ar@{^{(}->}[d]
  & \Sh(\thH, \thXH)      \ar@{^{(}->}[d]
\\  \tau \Sh(\gG, X)      \ar[r]^{\phi_{\tau,h,\gG}}
  & \Sh(\thG, \thX).
}}
\end{equation}
After translating notation, this is precisely the assertion in the proof of \cite[Lemma~9.5]{MS82b} that \( \phi_{\tau,h,\gG} \) maps \( \tau \Sh(\gH, X_\gH) \) into \( \Sh(\thH, \thXH) \).

For completeness, we prove this assertion.
Using equation~\eqref{eqn:phi-hg} for both \( \gG \) and~\( \gH \) and the fact that \( \theta_{\tau,h,\gG} \) restricts to \( \theta_{\tau,h,\gH} \), we get the following for every \( g \in \gH(\bA_f) \):
\begin{align*}
     \phi_{\tau,h,\gG}(\tau[h,g])
   = [{^\tau h}, \theta_{\tau,h,\gG}(g)]
   = [{^\tau h}, \theta_{\tau,h,\gH}(g)]
   = \phi_{\tau,h,\gH}(\tau[h,g]).
\end{align*}
Since \( \{ \tau[h,g] : g \in \gH(\bA_f) \} \) is dense in \( \tau \Sh(\gH, X_\gH) \), this shows that \eqref{inclusion-commutes} commutes.
\end{proof}

\begin{corollary}
Let \( S \) be the canonical model of a Shimura variety \( \Sh_K(\gG, X) \).
Let \( Z \subset S \) be a special subvariety, defined over \( \Qbar \).
Then for every \( \tau \in \Gal(\Qbar/E(\gG, X)) \), \( \tau(Z) \) is also a special subvariety of \( S \).
\end{corollary}

\begin{proof}
By definition, \( Z \) is a geometrically irreducible component of the image of \( T_g \circ [f] \) for some morphism of Shimura data \( f \colon (\gH, X_\gH) \to (\gG, X) \) and some Hecke correspondence \( T_g \) on \( \Sh_K(\gG, X) \).
%Now \( \tau(Z) \) is a component of the image of \( \tau T_g \circ \tau[f] \).
The Hecke correspondence \( T_g \) is defined over \( E(\gG, X) \) and every component of the image of a special subvariety under a Hecke correspondence is special.
Hence it suffices to show that each component of the image of \( \tau[f] \colon \Sh_{K_\gH}(\gH, X_\gH) \to \Sh_K(\gG, X) \) is a special subvariety.

By \cite[Prop.~4.3]{Pin05}, we may assume that \( f \) is injective.
Then \( \tau \Sh(\gH, X_\gH) \) is a Shimura sub-pro-variety of \( \Sh(\gG, X) \) by \cref{sp-subpro}.
Projecting down to finite level, we deduce that every geometrically irreducible component of the image of \( \tau \Sh_{K_\gH}(\gH, X_\gH) \) in \( \Sh_K(\gG, X) \) is indeed a special subvariety.
\end{proof}

\section{Conjugation and complexity} \label{sec:conj-complexity}

We conclude the paper by proving \cref{galois-complexity}.
We first need the following lemma on morphisms of Shimura varieties.

\begin{lemma} \label{zeta}
Let \( f \colon (\gG_1, X_1) \to (\gG_2, X_2) \) be an isomorphism of Shimura data and let \( \beta \in \gG_1(\bA_f) \).
Define a morphism of pro-varieties by
\[ \phi = [f] \circ T_\beta \colon \Sh(\gG_1, X_1) \to \Sh(\gG_2, X_2). \]
Let \( \theta \colon \gG_1(\bA_f) \to \gG_2(\bA_f) \) be the continuous group homomorphism
\( \theta(g) = f(\beta^{-1} g \beta) \).
Let \( \theta' \colon \gG_1(\bA_f) \to \gG_2(\bA_f) \) be any continuous group homomorphism.

Then
\[ \phi \circ T_g = T_{\theta'(g)} \circ \phi \text{ for all } g \in \gG_1(\bA_f) \]
if and only if \( \theta' \) has the form \( \theta'(g) = \zeta(g) \theta(g) \) for some continuous homomorphism \( \zeta \colon \gG_1(\bA_f) \to \ov{\gZ_2(\bQ)} \),
where \( \gZ_2 \) denotes the centre of \( \gG_2 \).
\end{lemma}

\begin{proof}
Let \( \zeta(g) = \theta'(g) \theta(g)^{-1} \).

For every \( h \in X_1 \) and \( a, g \in \gG_1(\bA_f) \), we can calculate
\begin{gather*}
   T_{\theta'(g)} \phi([h,a]) = T_{\theta'(g)} [f] T_\beta([h,a]) = [f_*(h), f(a\beta) \theta'(g)] = [f_*(h), f(a\beta) \zeta(g) \theta(g)],
\\ \phi T_g([h,a]) = [f] T_\beta T_g([h,a]) = [f_*(h), f(ag\beta)] = [f_*(h), f(a\beta) \theta(g)].
\end{gather*}
Since \( T_{\theta(g)} \) is invertible, we deduce that \( \phi \circ T_g = T_{\theta'(g)} \circ \phi \) if and only if
\begin{equation} \label{eqn:equiv}
[f_*(h), f(a\beta)] = [f_*(h), f(a\beta) \zeta(g)] \text{ for all } h \in X_1, a \in \gG_1(\bA_f).
\end{equation}

From the double quotient description~\eqref{eqn:spv-c-pts} of \( \Sh(\gG_2, X_2)(\bC) \), we see immediately that, if the image of \( \zeta \) lies in \( \ov{\gZ_2(\bQ)} \), then \eqref{eqn:equiv} holds.

Now for the converse.
Assume that \eqref{eqn:equiv} holds for every \( g \in \gG_1(\bA_f) \).
Choose a Hodge generic point \( h \in X_1 \).
Thanks to \eqref{eqn:spv-c-pts} and \eqref{eqn:equiv}, for every \( a,g \in \gG_1(\bA_f) \), there exist \( \gamma(a,g) \in \gG_2(\bQ) \) and \( \nu(a,g) \in \ov{\gZ_2(\bQ)} \) such that
\begin{align}
   \gamma(a,g) \, f_*(h) &= f_*(h),                                  \label{mu-fh}
\\ \gamma(a,g) \, f(a\beta) \, \nu(a,g) &= f(a\beta) \zeta(g).   \label{mu-fg}
\end{align}

According to \cite[Lemma~2.2]{UY13}, the intersection of \( \gG_2(\bQ) \) with the stabiliser of \( f_*(h) \) in \( \gG_2(\bR) \) is equal to the \( \bQ \)-points of the centraliser of \( \MT(f_*(h)) \), that is,
\begin{equation} \label{eqn:stab-q}
\gG_2(\bQ) \cap \Stab_{\gG_2(\bR)}(f_*(h)) = Z_{\gG_2}(\MT(f_*(h)))(\bQ).
\end{equation}
Since \( h \) is Hodge generic in \( X_1 \) and \( f \) is an isomorphism, \( f_*(h) \) is Hodge generic in~\( X_2 \).
By the axiom \cite[2.1.1.3]{Del79}, \( \gG_2^\der \subset \MT(X_2) \).
Hence
\begin{equation} \label{eqn:centralisers}
Z_{\gG_2}(\MT(f_*(h))) = Z_{\gG_2}(\MT(X_2)) \subset Z_{\gG_2}(\gG_2^\der).
\end{equation}
Since \( \gG_2 \) is reductive, \( Z_{\gG_2}(\gG_2^\der) = \gZ_2 \).
Hence \eqref{mu-fh}, \eqref{eqn:stab-q} and \eqref{eqn:centralisers} imply that \( \gamma(a,g) \in \gZ_2(\bQ) \).
Since \( \gamma(a,g) \) and \( \nu(a,g) \) both lie in \( \ov{\gZ_2(\bQ)} \subset \gZ_2(\bA_f) \), \eqref{mu-fg} gives
\[ \gamma(a,g) \nu(a,g) = \zeta(g). \]
We conclude that \( \zeta(g) \) lies in \( \ov{\gZ_2(\bQ)} \) and that \( \zeta \) is a group homomorphism.
\end{proof}

The following proposition is the main step in the proof of \cref{galois-complexity}.

\begin{proposition} \label{galois-special-mt}
Let \( (\gG, X) \) be a Shimura datum and let \( [h,g] \in \Sh(\gG, X) \) be a special point.
For every \( \tau \in \Gal(\Qbar/E(\gG, X)) \),
if \( \tau[h,g] = [h_\tau,g_\tau] \in \Sh(\gG, X) \), then
there exist:
\begin{enumerate}[(a)]
\item an isomorphism of \( \bQ \)-tori \( f_2 \colon \MT(h_\tau) \to \MT(h) \);
\item an automorphism \( \alpha \) of the topological group \( \MT(h)(\bA_f) \) such that
\( \alpha \circ f_2 \) maps \( g_\tau K g_\tau^{-1} \cap \MT(h_\tau)(\bA_f) \) onto \( gKg^{-1} \cap \MT(h)(\bA_f) \) for each compact open subgroup \( K \subset \gG(\bA_f) \).
\end{enumerate}
\end{proposition}

\begin{proof}
By \eqref{eqn:phi-iota} and \eqref{eqn:phi-hg}, we have
\[ \phi(\tau;h)[h_\tau,g_\tau] = \phi_{\tau,h}(\tau [h, g]) = [{^\tau h}, \theta_{\tau,h}(g)] \text{ in } \Sh(\thG, \thX). \]
Write \( \phi(\tau;h) = [f_1] \circ T_{\beta_1} \) as in \cite[Prop.~7.8]{MS82b}.
Recalling the double coset description \eqref{eqn:spv-c-pts} of \( \Sh(\thG, \thX)(\bC) \), we get
\begin{gather}
   f_{1*}(h_\tau) = \gamma \, {^\tau h},  \label{f-h}
\\ f_1(g_\tau \beta_1) = \gamma \theta_{\tau,h}(g)z. \label{f-g}
\end{gather}
for some \( \gamma \in \thG(\bQ) \) and \( z \in \ov{\thZ(\bQ)} \).

Thanks to \eqref{f-h}, \( f_2 := \Inn(\gamma^{-1}) \circ f_1 \) restricts to an isomorphism of \( \bQ \)-tori
\[ \MT(h_\tau) \to \MT({^\tau h}) = \MT(h). \]
This proves (a).

Define \( \theta_{\tau,h}' \colon \gG(\bA_f) \to \thG(\bA_f) \) by
\[ \theta_{\tau,h}'(g') = f_1(\beta_1^{-1} g' \beta_1). \]
Using \eqref{f-g}, we can calculate
\begin{align} \label{eqn:f2}
    f_2(g_\tau K g_\tau^{-1})
  & = \gamma^{-1} f_1(g_\tau \beta_1) \, f_1(\beta_1^{-1} K \beta_1) \, f_1(\beta_1^{-1} g_\tau^{-1}) \gamma   \notag
\\& = \theta_{\tau,h}(g) z \, \theta_{\tau,h}'(K) \, z^{-1} \theta_{\tau,h}(g)^{-1}.
\end{align}

Because \( \phi(\tau;h) \) satisfies property~(i) from p.~\pageref{phi-tau-h}, we can apply \cref{zeta} to establish that
\( \theta_{\tau,h}(g) = \zeta(g) \theta_{\tau,h}'(g) \) for some \( \zeta(g) \in \ov{\thZ(\bQ)} \).
Substituting this into~\eqref{eqn:f2} and using the fact that \( z \) and \( \zeta(g) \) are in \( \ov{\thZ(\bQ)} \subset \thZ(\bA_f) \), we get
\begin{equation} \label{eqn:f2-theta}
f_2(g_\tau K g_\tau^{-1}) = \theta_{\tau,h}'(gKg^{-1}).
\end{equation}

Let \( \gM = \MT(h) \subset \gG \) and \( \thM = \MT({^\tau h}) \subset \thG \).
Let \( \gT \subset \gG \) denote a maximal \( \bQ \)-torus such that \( h \) factors through \( \gT_\bR \) and let \( \thT = {^\tau\cS} \times^\cS \gT \subset \thG \).
Since \( \cS \) acts trivially on~\( \gT \), twisting by \( {^\tau\cS} \) gives a bijection between \( \bQ \)-algebraic subgroups of \( \gT \) and \( \bQ \)-algebraic subgroups of \( \thT \), as explained in \cref{thT}.
Hence \( \thM = {^\tau\cS} \times^\cS \gM \).

We claim that \( \theta'_{\tau,h} \) restricts to an isomorphism \( \gM(\bA_f) \to \thM(\bA_f) \).
In order to establish this, we look into the definition of \( f_1 \) and \( \beta_1 \) at \cite[p.~328--329]{MS82b}.
Choose any element \( a(\tau) \in {^\tau\cS}(\Qbar) \) and define an isomorphism \( f \colon \gG_{\Qbar} \to \thG_{\Qbar} \) by
\[ f(g) = a(\tau).g. \]
Let \( L \) be a number field over which \( f \) is defined.
Milne and Shih use the Taniyama group to construct an element
\[ \tilde{\beta}(\tau,h) \in \gT(\bA_f \otimes_\bQ L). \]
The cocycle \( \tilde\beta(\tau,h)^{-1} \cdot \sigma\tilde\beta(\tau,h) \) becomes trivial in \( \rH^1(\bQ, \gG) \), and hence it is the coboundary of some element \( v \in \gG(\Qbar) \).
Define \( f_1 \colon \gG \to \thG \) and \( \beta_1 \in \gG(\bA_f) \) by
\begin{gather*}
   f_1 = f \circ \Inn(v^{-1}),
\\ \beta_1 = \tilde\beta(\tau,h) \, v^{-1}.
\end{gather*}
From these descriptions of \( f_1 \) and \( \beta_1 \), we can read off
\begin{equation} \label{eqn:theta-beta}
\theta_{\tau,h}'(g') = f(\tilde\beta(\tau,h)^{-1} \, g' \, \tilde\beta(\tau,h))
\end{equation}
(as elements of \( \thG(\bA_f \otimes_\bQ L) \)) for all \( g' \in \gG(\bA_f) \).

Since \( \tilde\beta(\tau,h) \in \gT(\bA_f \otimes_\bQ L) \) and \( \gT \) is a torus containing \( \gM \), \( \tilde\beta(\tau,h) \) commutes with \( \gM(\bA_f) \).
Furthermore, since \( f \) is defined as twisting by an element of \( {^\tau\cS}(\Qbar) \), it maps \( \gM \) to \( \thM \).
Thus \eqref{eqn:theta-beta} shows that \( \theta_{\tau,h}' \) maps \( \gM(\bA_f) \) into \( \thM(\bA_f \otimes_\bQ L) \).
By definition, \( \theta_{\tau,h}' \) maps \( \gG(\bA_f) \) into \( \thG(\bA_f) \), so in fact it maps \( \gM(\bA_f) \) into \( \thM(\bA_f) \).

By the same argument, \( \theta_{\tau,h}'^{-1} \) maps \( \thM(\bA_f) \) into \( \gM(\bA_f) \).
Thus \( \theta_{\tau,h}' \) restricts to an isomorphism \( \gM(\bA_f) \to \thM(\bA_f) \).

Composing \( \theta_{\tau,h}'^{-1} \) with the identification \( \thM(\bA_f) = \gM(\bA_f) \) gives an automorphism \( \alpha \) of \( \gM(\bA_f) \).
Now \eqref{eqn:f2-theta} shows that
\[ f_2(g_\tau K g_\tau^{-1}) \cap \gM(\bA_f) = \alpha^{-1}(gKg^{-1} \cap \gM(\bA_f)). \]
Thus \( \alpha \) is the automorphism required for~(b).
\end{proof}

\begin{corollary*} (\cref{galois-complexity})
Let \( S = \Sh_K(\gG, X) \) be the canonical model of a Shimura variety over the reflex field \( E_\gG = E(\gG, X) \).
Let \( s \in S \) be a special point.
Then for every \( \tau \in \Gal(\Qbar/E_\gG) \),  we have \( \Delta(\tau(s)) = \Delta(s) \).
\end{corollary*}

\begin{proof}
Choose \( h \in X \) and \( g \in \gG(\bA_f) \) such that \( s = [h,g]_K \).
Let \( \tau[h,g] = [h_\tau, g_\tau] \in \Sh(\gG, X) \).
Then \( \tau(s) = [h_\tau, g_\tau]_K \).

We write \( \gM = \MT(h) \), \( D_\gM \), \( K_\gM^m \), \( K_\gM \) as in the definition of complexity of \( s \), and \( \gM_\tau = \MT(h_\tau) \), \( D_{\gM_\tau} \), \( K_{\gM_\tau}^m \), \( K_{\gM_\tau} \) for the analogous objects attached to \( \tau(s) \).

\Cref{galois-special-mt}(1) implies that the discriminants \( D_\gM \) and \( D_{\gM_\tau} \) are equal.

Let \( f_2 \colon \gM_\tau \to \gM \) and \( \alpha \colon \gM(\bA_f) \to \gM(\bA_f) \) be as in \cref{galois-special-mt}.
Since both \( f_2 \) and \( \alpha \) induce isomorphisms of topological groups on the adelic points, \( \alpha \circ f_2 \) maps \( K_{\gM_\tau}^m \) to \( K_\gM^m \).
By \cref{galois-special-mt}(2), \( \alpha \circ f_2 \) maps
\[ K_{\gM_\tau} = g_\tau K g_\tau^{-1} \cap \gM_\tau(\bA_f) \]
onto \( K_\gM = gKg^{-1} \cap \gM(\bA_f) \).
Consequently
\[ [K_{\gM_\tau}^m:K_{\gM_\tau}] = [K_\gM^m:K_\gM]. \]

The definition of complexity now tells us that \( \Delta(\tau(s)) = \Delta(s) \).
\end{proof}

\bibliographystyle{amsalpha}
\bibliography{excm}

\vspace{2em}
\hrule height 0.3pt width 10em
\vspace{1em}

\noindent Mathematics Institute, University of Warwick, Coventry CV4 7AL
England, U.K. 

\smallskip
\noindent {\tt martin.orr@warwick.ac.uk}

\end{document}